\documentclass[11 pt,a4paper]{article}
\usepackage{amsmath,amscd}
\usepackage{amstext}
\usepackage{amssymb,epsfig}
\usepackage{amsthm}
\usepackage{amsfonts}
\usepackage{graphicx}
\usepackage[all]{xy}
\usepackage{tikz}

\renewcommand{\footnote}{\endnote}
\newtheorem{theorem}{Theorem}[section]
\newtheorem*{Theorem}{Theorem}

\newtheorem{proposition}[theorem]{Proposition}
\newtheorem{corollary}[theorem]{Corollary}

\theoremstyle{definition}
\newtheorem{definition}[theorem]{Definition}
\newtheorem{remark}[theorem]{Remark}

\newtheorem{example}[theorem]{Example}

\def\proj#1{\pi_{#1}}
\usepackage[utf8]{inputenc}      

\begin{document}
\title{Regular simplices and periodic billiard orbits}
\author{Nicolas Bédaride\footnote{ Laboratoire d'Analyse Topologie et Probabilités  UMR 7353 , Université Aix Marseille, Centre de mathématiques et informatique 29 avenue Joliot Curie 13453 Marseille cedex, France. nicolas.bedaride@univ-amu.fr}\and Michael Rao\footnote{Laboratoire de l'Informatique du Parallélisme, équipe MC2, \'Ecole Normale Supérieure, 46 avenue d'Italie 69364 Lyon cedex 7.  michael.rao@ens-lyon.fr }}
\date{}
\maketitle

\begin{abstract} 
A simplex is the convex hull of $n+1$ points in $\mathbb{R}^{n}$ which form an affine basis. A regular simplex $\Delta^n$ is a simplex with sides of the same length. We consider the billiard flow inside a regular simplex of $\mathbb{R}^n$. We show the existence of two types of periodic trajectories. One has period $n+1$ and hits once each face. The other one has period $2n$ and hits $n$ times one of the faces while hitting once any  other face. In both cases we determine the exact coordinates for the points where the trajectory hits the boundary of the simplex.
\end{abstract}

\section{Introduction}
We consider the billiard problem inside a polytope. We start with
a point of the boundary of the polytope and we move along a
straight line until we reach again the boundary, where the trajectory is reflected 
according to the mirror law. A famous example of a
periodic trajectory is Fagnano's orbit: we consider an acute
triangle and the foot points of the altitudes. Those points define a
billiard trajectory which is periodic (see Figure \ref{fig1}). If we code the sides of the polygon by 
different letters, a billiard orbit becomes a word on this alphabet. 
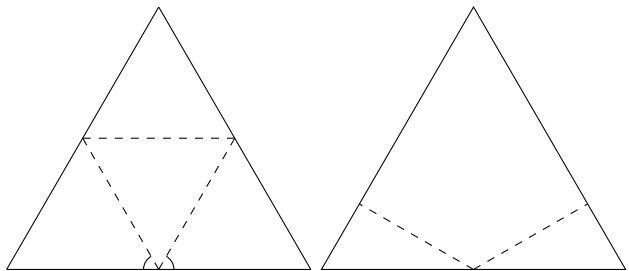
\begin{figure}
\begin{center}
\begin{tikzpicture}[scale=4]
\draw (0,0)--(1,0);
\draw (1,0)--(0.5,0.87);
\draw (0.5,0.87)--(0,0);
\draw (0.55,0) arc(0:60:0.05);
\draw (0.45,0) arc(180:120:0.05);
\draw[dashed] (0.5,0)--(0.25,0.435);
\draw[dashed] (0.5,0)--(0.75,0.435);
\draw[dashed] (0.25,0.435)--(0.75,0.435);
\end{tikzpicture}
\begin{tikzpicture}[scale=4]
\draw (0,0)--(1,0)--(0.5,0.87)--cycle;
\draw[dashed] (0.5,0)--(0.125,0.217);
\draw[dashed] (0.5,0)--(0.875,0.217);
\end{tikzpicture}
\end{center}
\caption{Periodic orbits coded $012$ and $0102$}\label{fig1}
\end{figure}
For the case of polygons in the plane some results are known. For example, one knows that
there exists a periodic orbit in each rational polygon (the angles
are rational multiples of $\pi$), and recently Schwartz,  \cite{Schw.05},  has proved
 the existence of a periodic billiard orbit in
every obtuse triangle with angle less than $100$ degrees. 
In \cite{Ga.Zv} it is proven that through every point of a right angled triangle passes a periodic orbit. 
A good survey of what is known about periodic orbits can be found in the
article \cite{Ga.St.Vo} by Galperin, Stepin and Vorobets or in the
book of Masur and Tabachnikov \cite{Ma.Tab}. For more recent results, see \cite{Schw.05}, \cite{Hoop.Schw.09} or \cite{Dav.Fuchs.Taba.11}.

In the polytope case much less is known. Consider a periodic billiard flow, then define the {\it boundary points} as the intersections of the trajectory with the boundary of the polytope. In \cite{Ga.Kr.Tr.95}  the study of the set of boundary points associated to a periodic word is made. 

There is no general result on periodic orbits; the only known result concerns the example of the tetrahedron. Stenman \cite{Sten} shows that a periodic orbit of length four 
exists in a regular tetrahedron. In \cite{Bed.08} the author proves that this orbit exists for an open set of tetrahedra near the regular tetrahedron. 
In both cases a boundary point with explicit coordinates is given, but the method can not be generalized to any dimension.

In this paper we consider the case of regular simplex in $\mathbb{R}^n$. 
A simplex is the convex hull of $n+1$ points in $\mathbb{R}^{n}$ which form an affine basis. A regular simplex $\Delta^n$ is a simplex with sides of the same length. We obtain by a short proof the existence of two periodic billiard trajectories in $\Delta^{n}$ for any $n\geq 2$. 
\begin{Theorem}
In a regular simplex $\Delta^{n}\subset \mathbb{R}^n$, there exists at least two periodic orbits:
\begin{itemize}
\item One has period $n+1$ and hits each face once.
\item The other has period $2n$ and hits one face $n$ times and hits each other face once.
\end{itemize}
\end{Theorem}

\section{Background on billiards}

Let $P$ be a polytope of $\mathbb{R}^n$, we will call {\it face} of $P$ the faces of maximal dimension. A {\it billiard orbit} is a (finite or infinite) broken line, denoted $\bold{x_0x_1}\dots$ such that 
\begin{enumerate}
\item Each $\bold{x_i}$ is a point in the interior of a face of $P$.
\item The unit vectors $\frac{\bold{x_{i+1}}-\bold{x_i}}{|\bold{x_{i+1}-x_i}|}$ and $\frac{\bold{x_{i+2}-x_{i+1}}}{|\bold{x_{i+2}-x_{i+1}}|}$ are symmetric with respect to the orthogonal reflection in the face containing $\bold{x_{i+1}}$ for every integer $i$. 
\end{enumerate}
For each integer $i$, $\bold{x_i}$ is called a {\it boundary point} of the orbit and  $\frac{\bold{x_{i+1}-x_i}}{|\bold{x_{i+1}-x_i}|}$ is the direction of the orbit at this boundary point.

A billiard orbit $\bold{x_0\dots x_n\dots}$ is called {\it periodic} of period $n$ if $n$ is the smaller integer such that for each integer $i\geq 0$, $\bold{x_{n+i}=x_i}$.

A coding of the billiard flow is given via a labelling of the faces of the polyhedron 
by elements of a finite set $\mathcal{A}$. Then to each nonsingular billiard orbit there is naturally assigned an {\it infinite word} $v=a_0a_1\dots,$ where $a_i\in\mathcal{A}$ is the label of the face which contains $x_i$. 

To each periodic billiard orbit of period $n$, we can associate an infinite periodic word. The word $a_0\dots a_{n-1}$ is called the {\it fundamental period} of the word. 

\section{Statement of results}
\subsection{Barycentric coordinates}
Consider $n+1$ points $\bold{a_0,\dots, a_n}\in\mathbb{R}^n$ which form an affine basis, then for every point $\bold{m}\in\mathbb{R}^n$ there exists real numbers $\lambda_0,\dots,\lambda_n$ such that 
$\displaystyle\sum_{0\leq i\leq n}\lambda_i\bold{a_i}=\bold{m}$ Throughout this paper $(\lambda_0,\dots,\lambda_n)$ will be called {\it barycentric coordinates} of $\bold{m}$. We do not assume that the sum of coordinates is $1$, in order to avoid lengthly formulas. Therefore there is no uniqueness of coordinates for a given point.

\subsection{Definitions and results}
\begin{definition}\label{def:mo}
Consider a regular simplex in $\mathbb{R}^n$ and consider barycentric coordinates with respect to the vertices. Define three points $\bold{m_0, p_1, r_1}\in\mathbb{R}^n$ by the following formulas, where $\proj{i}(m)$ denotes the $i$-th coordinate of $\bold{m}$. 

\begin{itemize}
\item $\proj{i}(m_0)=
-i^2+(n+1)i\quad 0\leq i\leq n$.

\item 
$\proj{i}(p_1)=
\begin{cases}
0 &\text{if $i=0$,}\\
-2(n+1)i^2+2(n+1)^2i-n(n+2)
&\text{if $0<i\le n$.}
\end{cases}$

\item
$\proj{i}(r_1)=
\begin{cases}
n &\text{if $i=0$,}\\
2(n+1)(n-i+1)(i-1)
&\text{if $0< i \le n$.}
\end{cases}$
\end{itemize}


         

\end{definition}

\begin{definition}\label{def:mi}
For a point in $\mathbb{R}^n$ with barycentric coordinates $(x_0,\dots,x_n)$ the cyclic right shift of these coordinates is the point with coordinates $(x_n,x_0,\dots,x_{n-1})$.
Define the points $\bold{m_i}, i=1\dots {n}$ in $\mathbb{R}^n$ obtained by cyclic right shift of the coordinates of $\bold{m_0}$.  Now the points $\bold{p_i}, i=2\dots {n}$ (resp $\bold{r_i}$) are obtained by a cyclic right shift of the last $n$ coordinates of $\bold{p_1}$ (resp. $\bold{r_1}$).
\end{definition}
\begin{example}
We have $\bold{m_1}=(n,0,n,2n-2\dots, 2n-2)$.
We have $\bold{p_2}=(0,n^2,n^2,3n^2-2n-4,\ldots,3n^2-2n-4).$
\end{example}

We obtain :
\begin{theorem}\label{thm:simplexe}
In a regular simplex $\Delta^{n}\subset \mathbb{R}^n$ one has:
\begin{itemize}
\item[(1)] The word $012\dots n$ describes the fundamental period of a periodic word that codes an orbit of period $n+1$ which passes through every face once during the period. The boundary points of this orbit are given by $\bold{m_i}, i=0\dots n$.

\item[(2)] The word $010203\dots 0n$ describes the fundamental period of a periodic word that codes an orbit of period $2n$ which passes through one face $n$ times while hitting once time any other face. The $n$ boundary points on the face labeled $0$ have coordinates given by $\bold{p_i}, i=1\dots n$. 
The other points are given by $\bold{r_i}, i=1\dots n$.
\end{itemize}
\end{theorem}

\begin{corollary}
The boundary points of the periodic trajectories satisfies the following properties:
\begin{itemize}
\item If $n$ is odd, then $\bold{m_i}$ and $\bold{p_i}$ are situated at the intersection of $\frac{n+1}{2}$ hyperplanes for every integer $i$, when each hyperplane is orthogonal to an edge of the simplex and passes through the middle of this edge.
\item If $n$ is even, the points lie on a line segment that connects a vertex to the intersection of $\frac{n}{2}$ hyperplanes.
\end{itemize}
\end{corollary}
The corollary can be deduced directly from the properties of the barycentric coordinates of the periodic points. A formal proof is left to the reader.

\begin{remark}\label{rem1}
The isometric group of a regular tetrahedron $\Delta^n$ is the permutation group $\mathfrak{S}_{n+1}$. If $v$ is a periodic billiard word and $\sigma$ a permutation in $\mathfrak{S}_{n+1}$, then $\sigma(v)$ is also a periodic word. Moreover, if $v$ is a periodic word, then the shift of this word is also a periodic word corresponding to the same periodic orbit. Thus our theorem gives the existence of $\frac{(n-1)!}{2}$ points of first type and $\frac{(2n)!}{2n!}$  periodic points of the second type inside the regular simplex. \end{remark}
\subsection{Examples}
To give concrete example we consider the cases of $\Delta^2,\Delta^3,\Delta^4$: 
\begin{itemize}
\item $$\bold{m_0}=\begin{cases}
(0,2,2) & \text{if $n=2$}\\
(0,3,4,3) & \text{if $n=3$}\\
(0,2,3,3,2) & \text{if $n=4$}\\ 

\end{cases}$$
$$\bold{p_1}=\begin{cases} 
(0,4,4) & \text{if $n=2$}\\
(0,9,17,9) & \text{if $n=3$}\\
(0,16,36,36,16) & \text{if $n=4$}\\
(0,25,61,73,61,25)  & \text{if $n=5$}
\end{cases}$$
$$\bold{r_1}=\begin{cases}
(2,0,6) & \text{if $n=2$}\\
(3,0,16,16) & \text{if $n=3$}\\
(4,0,30,40,30) & \text{if $n=4$}\\
(5,0,48,72,72,48)  & \text{if $n=5$}
\end{cases}$$
\item In $\Delta^2$ the words of first and second type are $012$ and $0102$. The boundary point for the first one is the middle of the edge. For the second one, every point on the edge is a boundary point. This case is an extremal case since each point on the edge $0$ is a periodic point. Moreover the direction of the billiard orbit is orthogonal to the edges $1$ and $2$ of the simplex (see Figure \ref{fig1}).

\item In $\Delta^3$, the two fundamental periods of the periodic words are $0123$ and $010203$. 
Each boundary point of the periodic trajectories is on an height of a triangular face. 
Indeed the point $(3,3,4)$ is the barycenter of one vertex of the triangle and one midpoint of an edge. 
It lies on a segment joining one vertex to the midpoint of the opposite edge. Thus it is on a height of the simplex (see Figure \ref{fig2}). 

\item In $\Delta^4$ for the fundamental period $01234$, each boundary point is inside a regular tetrahedron, it is on a segment which links the two midpoints of non coplanar edges of the tetrahedron. 
The point $(2,3,3,2)$ is the barycenter of two middles of edges. This segment is orthogonal to the two edges. The symmetric point $(3,2,2,3)$ is also on this segment. There are six points obtained by permutation. They are on three edges passing through the center of the regular simplex (see Figure \ref{fig2}).
\end{itemize}
\begin{figure}
\begin{center}
\begin{tikzpicture}
\draw (0,0)--(2,0);
\draw (1,0) node{$\bullet$};
\end{tikzpicture}
\begin{tikzpicture}[scale=2.3]
\draw (0,0)--(1,0);
\draw (1,0)--(0.5,0.87);
\draw (0.5,0.87)--(0,0);
\draw[green] (0.5,0)--(0.5,0.87);
\draw[green] (0,0)--(0.75,0.435);
\draw[green] (1,0)--(0.25,0.435);
\draw (0.5,0.17) node{$\bullet$};
\draw (0.6,0.34) node{$\bullet$};
\draw (0.4,0.34)  node{$\bullet$};
\end{tikzpicture}
\begin{tikzpicture}[scale=1.4]
\draw (0,0)--(2,0);
\draw (2,0)--(1,2);
\draw (1,2)--(0,0);
\draw[dashed] (0,0)--(1.5,0.4);
\draw[dashed] (2,0)--(1.5,0.4);
\draw[dashed] (1,2)--(1.5,0.4);
\draw[green] (1,0)--(1.25,1.2);
\draw[green] (0.75,0.2)--(1.5,1);
\draw[green] (1.75,0.2)--(0.5,1);
\draw (1.05,0.22)node{$\bullet$};
\draw (1.2,0.96)node{$\bullet$};
\draw (0.9,0.36)node{$\bullet$};
\draw (1.35,0.84 )node{$\bullet$};
\draw (1.5,0.36)node{$\bullet$};
\draw (0.71,0.84)node{$\bullet$};
\end{tikzpicture}
\end{center}
\caption{Boundary points in a face of $\Delta^n$ for the first periodic billiard orbit with $n=2,3,4$}\label{fig2}
\end{figure}
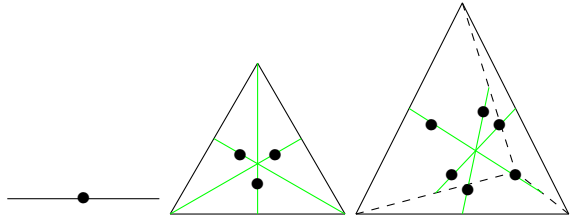


\section{Proof of Theorem \ref{thm:simplexe}}
\subsection{Part (1)}

Let $i$ be an integer in $[0\dots n]$. Barycentric coordinates of the points $\bold{m_{i-1}}$, $\bold{m_i}$ and $\bold{m_{i+1}}$ are given by Definition \ref{def:mo} and Definition \ref{def:mi}. (Indices are taken modulo $n+1$.) We must show that the image of $\bold{m_{i-1}}$ by the reflexion through face $i$ belongs to the line $(\bold{m_{i}m_{i+1}})$. By symmetry we can restrict to the case $i=0$.

Let $\bold{m'_n}$ denote the image of $\bold{m_n}$ under reflection in the hyperface $0$. We compute the barycentric coordinates of $\bold{m'_n}$. To do so, consider the orthogonal reflection of vertex $\bold{P_0}$ of $\Delta^n$ through the hyperface $0$. Denote it $\bold{P_0'}$. If $\bold{m_{n}}$ is barycentric of $\bold{P_0\dots P_{n}}$, then $\bold{m'_n}$ is barycentric with same coefficients of $\bold{P'_0,P_1,\dots P_{n}}$. We obtain $\bold{P_0'}=(-1,2/n,2/n,\dots,2/n)$ since the middle of the segment $[\bold{P_0P_0'}]$ is the center of the face labeled $0$. By definition: 
$$
\proj{i}(m_0)=
\begin{cases}
0 & \text{if $i=0$,} \\
-i^2+(n+1)i & \text{if $0<i<n$,} \\
n & \text{if $i=n$.} 
\end{cases} 
$$
We give an equivalent formula for coordinates of $m_1$:
$$
\proj{i}(m_1)=
\begin{cases}
n & \text{if $i=0$,} \\
-i^2+(n+3)i-n-2 & \text{if $0<i<n$,} \\
2n-2 & \text{if $i=n$.} 
\end{cases} 
$$
$$
\proj{i}(m_n)=
\begin{cases}
n & \text{if $i=0$,} \\
-i^2+(n-1)i+n & \text{if $0< i<n$,} \\
0 & \text{if $i=n$.} 
\end{cases} 
$$
Since $\bold{P'_0}=(-1,\frac{2}{n},\ldots,\frac{2}{n})$, we have:
$$
\proj{i}(m'_n)=
\begin{cases}
-n & \text{if $i=0$,} \\
-i^2+(n-1)i+n+2 & \text{if $0< i<n$,} \\
2 & \text{if $i=n$.} 
\end{cases} 
$$
It suffices to remark that $\bold{m_0}=(\bold{m_1}+\bold{m'_{n}})/2$ to finish the proof. 

\begin{remark}
A periodic billiard orbit is a polygonal path inside the regular simplex. For the periodic word of fundamental period $012\dots n$ we see that each edge of such a path has the same length. But the $n+1$ boundary points of such a periodic orbit do not form a regular simplex for each integer $n$.
\end{remark}

\subsection{Part (2)}
By symmetry, we can restrict ourself to two cases. We show that the image $\bold{r'_n}$ by reflection through face $0$ of the point $\bold{r_{n}}$  belongs to the line $(\bold{p_1r_1})$, and that the image $\bold{p'_2}$ by reflection through face $1$ of the point $\bold{p_{2}}$  belongs to the line $(\bold{p_1r_1})$.

$$
\proj{i}(r_1)=\begin{cases}
n &\text{if $i=0$,}\\
2(n+1)(n-i+1)(i-1)
&\text{if $0< i \le n$.}
\end{cases}
$$

$$
\proj{i}(r_n)=\begin{cases}
n &\text{if $i=0$,}\\
2(n+1)(n-i)i
&\text{if $0< i \le n$.}
\end{cases}
$$

$$
\proj{i}(r'_n)=\begin{cases}
-n &\text{if $i=0$,}\\
2(n+1)(n-i)i+2
&\text{if $0< i \le n$.}
\end{cases}
$$

$$
\proj{i}(p_1)=\begin{cases}
0 &\text{if $i=0$,}\\
-2(n+1)i^2+2(n+1)^2i-n(n+2)
&\text{if $0<i\le n$.}
\end{cases}
$$

$$
\proj{i}(p_2)=\begin{cases}
0 &\text{if $i=0$,}\\
n^2 &\text{if $i=1$,}\\
-2(n+1)(i-1)^2+2(n+1)^2(i-1)-n(n+2)&\text{if $1<i\le n$.}
\end{cases}
$$

Since $\bold{P'_1}=(\frac{2}{n},-1,\frac{2}{n},\ldots,\frac{2}{n})$, we have:
$$
\proj{i}(p'_2)=\begin{cases}
2n &\text{if $i=0$,}\\
-n^2 &\text{if $i=1$,}\\
-2(n+1)(i-1)^2+2(n+1)^2(i-1)-n(n+2)+2n&\text{if $1<i\le n$.}
\end{cases}
$$

Thus $\bold{p_1=(r_1+r'_{n})/2}$ and $\bold{r_1=(p_1+p'_{2})/2}$. This concludes the proof of Theorem~\ref{thm:simplexe}.


\section{Algorithm to find periodic orbits}
This section is not necessary to obtain proof of Theorem \ref{thm:simplexe}. Nevertheless we explain our method to find the coordinates of the periodic point involved in the theorem.

Let $v$ be a billiard word, we denote by $s_{v_i}$ and $S_{v_i}$ the affine and vectorial reflection through the face labeled $v_i$. Then $s_v, S_v$ are defined as product of maps $s_{v_i}$ or $S_{v_i}$ for $i=0\dots n$.
First we recall a proposition of \cite{Bed.08} :
\begin{proposition}\label{percrucial}
Let $P$ a polyhedron. then the following properties are equivalent:\\
(1) Ther exists a word $v$ which is the prefix of a periodic word with period $|v|$.\\
(2) There exists $\bold{m}\in v_0$ such that
$\overrightarrow{\bold{s_{v}(m)m}}$ is admissible with boundary point $\bold{m}$ for $vv_{0}$,
and $\bold{u}=\overrightarrow{\bold{s_{v}(m)m}}$ is such that $S_v\bold{u}=\bold{u}$.
\end{proposition}

Now we explain the algorithm. Let $v$ be a word. To prove that $v$ is a periodic word, we must find a point $\bold{m}$ in a face of the simplex and a direction $\bold{u}$ such that the billiard orbit of $(\bold{m},\bold{u})$ is periodic with coding $v$. Thus we must compute the eigenvector associated to $S_v$ for the eigenvalue $1$, and find $\bold{m}$ on the axis of $s_v$. In classical coordinates systems, the vertices of a regular simplex have simple formulas, thus we can compute $\bold{u}$. 

Now barycentric coordinates are useful to find $\bold{m}$.
Indeed the image of a point $\bold{m}$ by a reflection through a face of the simplex has barycentric coordinates which depends linearly over $\mathbb{Q}$ of those of $\bold{m}$: it is the barycenter with same coefficients of the vertices of the simplex except the vertex which does not belong to the face. This vertex is replaced by the image of the vertex by the reflection. This image can be expressed as barycenter with rational coefficients. Finally we see that $\overrightarrow{\bold{s_{v}(m)m}}$ is a linear combination of coordinates of $\bold{m}$. Thus coordinates of $m$ are easy to find.

The last step is to check that $\bold{m}$ and all other points are interior to the faces labelled $v_i, i=0\dots $. Of course, this can be simplified if there exists a coordinate system where the vertices of the regular simplex are in $\mathbb{Q}^n$, for example in dimension three such a system is given in \cite{Bed.08}. In this case $\bold{u}$ belongs to $\mathbb{Q}^3$. This is not possible in any dimension.

\section{Additional properties}
\subsection{First part}
In this section we consider the periodic words of Theorem \ref{thm:simplexe}.  As explained in Remark \ref{rem1} we can associate to a periodic word several boundary points inside the same face.
 We want to describe the convex hull generated by these points. Let us denote by $Q_n$ this polytope associated to $\Delta^n$. 
The motivation of this study is given by the example of lowest dimensions. 
The study of the polytope $Q$ associated to each periodic word could be a good way to try to obtain other periodic words.

\begin{proposition}
For the periodic orbit passing one time through each face, we have
\begin{itemize}
\item In dimension $2$, the polytope $Q_2$ is one point: the center of the edge. 
\item In dimension $3$, the polytope $Q_3$ is a regular triangle. 
\item In dimension $4$, the polytope $Q_4$ is a regular octahedron. 
\end{itemize}
\end{proposition}
\begin{proof}
The theorem gives barycentric coordinates of vertices of the boundary points of the periodic words. We use combinatorics to find description of the convex hull of these points.
\begin{itemize}
\item In dimension three, the points are $(4,3,3);(3,4,3);(3,3,4)$. They are on the heights of the triangle, and by symmetry the lengths are equal.
\item In dimension four the points are the permutations of $(2,3,3,2)$. The faces of the polyhedron defined as convex hull are given by position of one letter. Each face has three vertices.  There are $8$ faces. The faces images by transposition $2\leftrightarrow 3$ are parallel.
\end{itemize}
\end{proof}
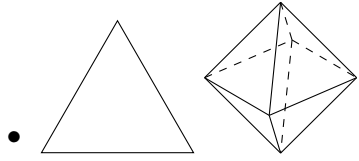
\begin{figure}
\begin{center}
\begin{tikzpicture}
\draw (0,0) node{$\bullet$};
\end{tikzpicture}
\begin{tikzpicture}
\draw (0,0)--(2,0)--(1,1.75)--cycle;
\end{tikzpicture}
\begin{tikzpicture}[scale=0.5]
\draw (-2,0)--(0,-2);
\draw (-2,0)--(-0.3,-1);
\draw (-0.3,-1)--(2,0);
\draw[dashed] (-2,0)--(0.3,1);
\draw[dashed] (0.3,1)--(2,0);
\draw[dashed] (0.3,1)--(0,2);
\draw(0,2)--(-0.3,-1);
\draw (-0.3,-1)--(0,-2);
\draw (0,-2)--(2,0);
\draw (2,0)--(0,2);
\draw (0,2)--(-2,0);
\draw [dashed](0.3,1)--(0,-2);
\end{tikzpicture}
\end{center}
\caption{Polytopes $Q_n$}
\end{figure}

In dimension $n\geq 5$, a simple computation shows that $Q_n$ is not a regular polytope. Nevertheless the polytope has nice properties.
For example in dimension five:  one vertex has coordinates: $(5,5,8,8,9)$, there are $30$ points. There are $15$ faces of dimension three. The permutations of $(5,5,8,8)$ form a regular octahedron. The permutations of $(5,5,8,9)$ form a regular tetrahedron truncated at two symmetric vertices: The faces are either triangles either hexagons. The permutations of $(5,8,8,9)$ form a polyhedron with $12$ vertices, $14$ faces and $24$ edges. The faces are quadrilateral or triangles.

\begin{figure}
\begin{center}
\begin{tikzpicture}[scale=0.8]
\draw (-2,0)--(-0.3,-1);
\draw (-0.3,-1)--(2,0);
\draw[dashed] (-2,0)--(0.3,1);
\draw[dashed] (0.3,1)--(2,0);
\draw[dashed] (0.3,1)--(0.08,1.75);
\draw (-0.5,1.5)--(-0.075,1.25)--(0.505,1.5);
\draw (-0.5,1.5)--(0.08,1.75);
\draw (0.08,1.75)--(0.505,1.5);
\draw (-0.5,-1.5)--(-0.075,-1.75)--(0.505,-1.5);
\draw[dashed] (-0.5,-1.5)--(0.08,-1.25);
\draw[dashed] (0.08,-1.25)--(0.505,-1.5);
\draw (-2,0)--(-0.5,1.5);
\draw(-0.3,-1)--(-0.075,1.25);
\draw (2,0)--(0.505,1.5);
\draw[dashed] (0.3,1)--(0.08,-1.25);
\draw (-2,0)--(-0.5,-1.5);
\draw (-0.3,-1)--(-0.075,-1.75);
\draw (2,0)--(0.505,-1.5);
\draw (0,-3) node[below]{$(5,8,8,9)$};
\end{tikzpicture}
\begin{tikzpicture}[scale=0.8]
\draw (-1.5,0.5)--(-0.5,1.5);
\draw (1.5,0.5)--(0.505,1.5);
\draw (-1.5,-0.5)-- (-0.5,-1.5);
\draw (0.505,-1.5)--(1.5,-0.5);
\draw (-0.5,1.5)--(-0.075,1.25)--(0.505,1.5);
\draw (-0.5,1.5)--(0.08,1.75);
\draw (0.08,1.75)--(0.505,1.5);
\draw (-0.5,-1.5)--(-0.075,-1.75)--(0.505,-1.5);
\draw[dashed] (-0.5,-1.5)--(0.08,-1.25);
\draw[dashed] (0.08,-1.25)--(0.505,-1.5);
\draw (-1.5,0.5)--(-1.575,-0.25);
\draw[dashed] (-1.5,0.5)--(-1.425,0.25);
\draw(-1.5,-0.5)--(-1.575,-0.25);
\draw[dashed] (-1.5,-0.5)--(-1.425,0.25);
\draw (1.5,-0.5)--(1.425,-0.25)--(1.5,0.5)--(1.575,0.25)--cycle;
\draw[dashed] (-1.425,0.25)--(-0.225,0.75);
\draw[dashed] (0.725,0.75)--(1.425,0.25);
\draw[dashed] (0.225,0.25)--(0.08,-1.25);
\draw[dashed] (0.225,1.25)--(0.08,1.75);
\draw (-0.725,-0.75)--(-1.575,-0.25);
\draw (-0.225,-0.25)--(-0.075,1.25);
\draw (0.225,-0.75)--(1.425,-0.25);
\draw (-0.225,-1.25)--(-0.075,-1.75);
\draw[dashed] (0.725,0.75)--(0.225,0.25)--(-0.225,0.75)--(0.225,1.25)--cycle;
\draw (-0.725,-0.75)--(-0.225,-0.25)--(0.225,-0.75)--(-0.225,-1.25)--cycle;
\draw (0,-3) node[below]{$(5,8,6,9)$};
\end{tikzpicture}
\end{center}
\caption{Polyhedras associated to two barycenters}
\end{figure}

\begin{proposition}
Consider the boundary points associated to the second periodic billiard word in Theorem \ref{thm:simplexe}. There is one of the faces which contains $n+1$ points. The convex hull generated by these points is a polytope similar to $Q_n$.
\end{proposition}

\subsection{Stability}
A natural question is the existence of periodic words in a non regular simplex. A good method to find such periodic orbits is to use stability.
A periodic orbit is said to be {\it stable} if it survives in a small perturbation of the polytope. There is a characterization of stable periodic words in dimension two in \cite{Ga.St.Vo}. In dimension three, a sufficient condition was given in \cite{Bed.08}. There is no generalization to higher dimensions, so that we can not prove the existence of periodic words in an $n$-dimensional simplex. Nevertheless, a generalization could appear to be natural and derivable.

The condition of stability in \cite{Bed.08} is given with the matrix $S_v$ defined in Section 5. If this matrix is different from the identity then the periodic trajectory is stable. The billiard paths obtained in $\Delta^3$ in the theorem are stable.

\bibliographystyle{alpha}
\bibliography{biblio-periodeBR}

\def\cprime{$'$} \def\cprime{$'$}
\begin{thebibliography}{GKT95}

\bibitem[Bé08]{Bed.08}
N.~Bédaride.
\newblock Periodic billiard trajectories in polyhedra.
\newblock {\em Forum Geom.}, 8:107--120, 2008.

\bibitem[DFT11]{Dav.Fuchs.Taba.11}
Diana Davis, Dmitry Fuchs, and Serge Tabachnikov.
\newblock Periodic trajectories in the regular pentagon.
\newblock {\em Mosc. Math. J.}, 11(3):439--461, 629, 2011.

\bibitem[GKT95]{Ga.Kr.Tr.95}
G.~Gal{\cprime}perin, T.~Kr{\"u}ger, and S.~Troubetzkoy.
\newblock Local instability of orbits in polygonal and polyhedral billiards.
\newblock {\em Comm. Math. Phys.}, 169(3):463--473, 1995.

\bibitem[GSV92]{Ga.St.Vo}
G.~A. Gal{\cprime}perin, A.~M. St{\"e}pin, and Ya.~B. Vorobets.
\newblock Periodic billiard trajectories in polygons: generation mechanisms.
\newblock {\em Uspekhi Mat. Nauk}, 47(3(285)):9--74, 207, 1992.

\bibitem[GZ03]{Ga.Zv}
G.~Galperin and D.~Zvonkine.
\newblock Periodic billiard trajectories in right triangles and right-angled
  tetrahedra.
\newblock {\em Regul. Chaotic Dyn.}, 8(1):29--44, 2003.

\bibitem[HS09]{Hoop.Schw.09}
W.~Patrick Hooper and Richard~Evan Schwartz.
\newblock Billiards in nearly isosceles triangles.
\newblock {\em J. Mod. Dyn.}, 3(2):159--231, 2009.

\bibitem[MT02]{Ma.Tab}
H.~Masur and S.~Tabachnikov.
\newblock Rational billiards and flat structures.
\newblock In {\em Handbook of dynamical systems, Vol.\ 1A}, pages 1015--1089.
  North-Holland, Amsterdam, 2002.

\bibitem[Sch06]{Schw.05}
R.~Schwarz.
\newblock Obtuse triangular billiards i: Near the (2,3,6) triangle.
\newblock {\em Journal of Experimental Mathematics}, 15(2), 2006.

\bibitem[Ste75]{Sten}
F.~Stenman.
\newblock Periodic orbits in a tetrahedral mirror.
\newblock {\em Soc. Sci. Fenn. Comment. Phys.-Math.}, 45(2-3):103--110, 1975.

\end{thebibliography}

\end{document}